\documentclass[11pt,a4paper]{article}  
   
\usepackage{amsmath}
\usepackage{amssymb} 
\usepackage{amsthm} 
\usepackage{color}
\usepackage{epsfig}
\usepackage{enumitem}
\usepackage[mathcal]{eucal}
\usepackage{mathrsfs}
\usepackage{url}

\definecolor{myred}{RGB}{255,50,70}
\definecolor{myblack}{RGB}{0,0,0}

\theoremstyle{plain}
\newtheorem{theorem}{Theorem}[section]
\newtheorem{lemma}[theorem]{Lemma}
\newtheorem{corollary}[theorem]{Corollary}

\theoremstyle{definition}
\newtheorem{definition}[theorem]{Definition}
\newtheorem{example}[theorem]{Example}
\newtheorem{algorithm}[theorem]{Algorithm}
\theoremstyle{remark}

\newcommand{\R}{\mathbb{R}}                     
\renewcommand{\Re}{\mathbb{R}}                  
\newcommand{\T}{\top\hspace{-1pt}}              
\newcommand{\argmin}{\mathop{\mathrm{argmin}}}  
\newcommand{\inner}[2]{\langle{#1},{#2}\rangle} 

\begin{document}  


\title{A barrier-type method for multiobjective optimization%
  \thanks{This work was supported by Grant-in-Aid for Young Scientists
    (B) (26730012) from Japan Society for the Promotion of Science,
    and a Grant (307679/2016-0) from National Counsel of Technological
    and Scientific Development (CNPq).}
}

\author{Ellen H. Fukuda%
  \thanks{Department of Applied Mathematics and Physics, 
    Graduate School of Informatics, Kyoto University, 
    Kyoto, 606-8501, Japan (\texttt{ellen@i.kyoto-u.ac.jp}).}
  \and
  L.~M. Gra\~na Drummond%
  \thanks{Faculty of Business and Administration, 
    Federal University of Rio de Janeiro, 
    Rio de Janeiro 22290-240, Brazil
    (\texttt{bolsigeno@gmail.com}).}
  \and 
  Fernanda M.~P. Raupp%
  \thanks{National Laboratory for Scientific Computing, 
    Rio de Janeiro 25651-075, Brazil 
    (\texttt{fernanda@lncc.br}).} 
} 

\date{March 29, 2018}  

\maketitle 

\begin{abstract}
  \noindent For solving constrained multicriteria problems, we introduce the
  multiobjective barrier method (MBM), which extends the scalar-valued
  internal penalty method. This multiobjective version of the
  classical method also requires a penalty barrier for the feasible
  set and a sequence of nonnegative penalty parameters. Differently
  from the single-valued procedure, MBM is implemented by means of an
  auxiliary ``monotonic'' real-valued mapping, which may be chosen in
  a quite large set of functions. Here, we consider problems with
  continuous objective functions, where the feasible sets are defined
  by finitely many continuous inequalities. Under mild assumptions,
  and depending on the monotonicity type of the auxiliary function, we
  establish convergence to Pareto or weak Pareto optima. Finally, we
  also propose an implementable version of MBM for seeking local
  optima and analyze its convergence to Pareto or weak Pareto
  solutions.\\

  \noindent \textbf{Keywords:} Barrier methods, multiobjective optimization,
  Pareto optimality, penalty methods.\\
\end{abstract}


\section{Introduction}       

Practical issues on different areas, such as statistics~\cite{CF98},
engineering~\cite{EKO90,SNR15}, environmental analysis~\cite{LWV92},
space exploration~\cite{Tav04}, management
science~\cite{GMNPT92,Whi98} and design~\cite{FD04} can be modeled as
constrained multicriteria minimization problems. There are many
procedures for solving these problems, including scalarization
techniques and heuristics. Among the scalarization approach, the
weighting method is possibly the most well-known. It basically
replaces the original multiobjective problem into a minimization of
some convex combination of the objectives. Its main drawback is the
fact that we do not know a priori which are the weights that do not
lead to unbounded scalar problems. Heuristics may neither guarantee
convergence. To overcome such drawbacks, extensions to the
vector-valued setting of classical real-valued methods have been
proposed in recent years~\cite{FGD14}.

In this work, we propose an extension of the \emph{internal penalty}
(or \emph{barrier-type}) method~\cite{Ber16,BSS06}, which is a
well-known technique for constrained scalar optimization problems with
continuous objective and constraints functions. As we know, the
scalar-valued method consists in sequentially minimizing, without
constraints, the objective with the addition of positive multiples of
a feasible set barrier. That is to say, one substitutes a constrained
problem by a sequence of unconstrained ones, for which we do have
efficient solving techniques. Under reasonable assumptions, the
sequence of those minimizers converges to an optimum for the original
constrained problem.

Here, we present a vector-valued version of this procedure, that we
call \emph{multiobjective barrier method} (MBM). Each iteration
requires solving an unconstrained scalar-valued problem. As in the
classical barrier method, the iterates generated are feasible, and the
penalized unconstrained objective values decreases
iteratively. Moreover, all accumulation point, if any, are optimal
solutions of the original problem. Furthermore, the convergence
conditions are generalizations of those required in the real-valued
case, and no additional hypotheses are needed. Besides the
nonincreasing convergent to zero parameter sequence of positive real
numbers and the barrier function, the vector-valued method uses an
auxiliary monotonic function (whose presence in the scalar case is
immaterial) and, depending on its type of monotonicity, the method
converges to Pareto or weak Pareto optimal solutions.

MBM shares some properties with other classical scalar-valued methods
extensions, as the steepest descent~\cite{FS00,GDS05}, the projected
gradient~\cite{GDI04,FGD11}, the Newton~\cite{FGDS09,GDRS14}, the
proximal point~\cite{BIS05}, and the external penalty
methods~\cite{FGDR16}. More precisely, all iterates of these
procedures are computed by solving scalar optimization subproblems,
and each one of them could be obtained by the application of the
corresponding scalar-valued method to a certain (a priori unknown)
linear combination of the objectives. These methods (including MBM)
seek just a single optimal point and the convergence results are
natural extensions of their scalar correlatives. Nevertheless, as
shown through numerical experiments (see, for instance,
\cite{FGDS09}), by initializing them with randomly chosen points, in
some cases, one can expect to obtain quite good approximations of the
optimal set.

The outline of this article is the following. In
Section~\ref{sec:prob}, we introduce the problem and the notion of
multiobjective barrier function, and give some examples. In
Section~\ref{sec:met}, we define the auxiliary functions, the
multicriteria barrier method, and make some comments about the
approach. In Section~\ref{sec:conv}, we analyze the convergence of
MBM, and verify that the results are extensions of the classical
single objective case. In Section~\ref{sec:app}, we exhibit a very
simple instance of the multiobjective constrained problem for which
the weighting method fails to obtain optimal solutions for an
arbitrary large family of parameters, while MBM furnishes a Pareto
optimal point for the problem. We also show an (\emph{ad hoc}) example
in which, by varying a parameter in the barrier, one can retrieve the
whole Pareto optimal set. Finally, in Section~\ref{sec:final} we draw
some conclusions and mention some possible continuations for this
work.


\section{The multicriteria problem and the internal penalty function}
\label{sec:prob}

We start this section with some notations which will be used in the
whole paper. The Euclidean inner product is written as
$\inner{\cdot}{\cdot}$. For any matrix $A$, its transpose is denoted
by $A^\T$. A vector $x \in \Re^n$ with entries $x_i \in \Re$, $i =
1,\dots,n$, is written as $(x_1,\dots,x_n)^\T$. A sequence of vectors
$y^1,y^2,\dots$ is denoted by $\{ y^k \}$. Also, we define the vector
with all ones with appropriate dimension as $e :=
(1,\dots,1)^\T$. Moreover, the gradient of a function $\zeta \colon
\Re^n \to \Re$ at $x \in \Re^n$ is denoted by $\nabla \zeta(x)$.

Now, consider $\Re^m$ endowed with the partial order induced by the
Paretian cone $\R^m_+ = \R_+ \times \dots \times \R_+$, where $\R_+ =
[0,+\infty)$, given by
\begin{displaymath}
  u \le v \quad \mbox{if} \quad u_j \le v_j \quad \mbox{for all} \quad
  j = 1,\dots,m,  
\end{displaymath} 
and with the following stronger relation:
\begin{displaymath}
  u < v \quad \mbox{if} \quad u_j < v_j \quad \mbox{for all} \quad j =
  1,\dots,m,
\end{displaymath}
with $u, v \in \R^m$. 

Given continuous functions $f \colon \R^n \to \R^m$ and $g \colon \R^n
\to \R^p$, define
\begin{displaymath}
  D := \{ x \in \R^n : g(x) \leq 0\}
\end{displaymath}
and consider the following constrained vector-valued optimization
problem:
\begin{equation}
  \label{eq:prob}
  \begin{array}{ll}
    \mbox{minimize} & f(x) \\  
    \mbox{subject to} & x \in D,
  \end{array}
\end{equation}  
understood in the Pareto or weak Pareto sense. Recall that a point
$x^* \in D$ is a {\it weak Pareto optimal} solution of the above
problem if there does not exist $x \in D$ such that $f(x) <
f(x^*)$. Also, $x^* \in D$ is a {\it Pareto optimal} solution of
\eqref{eq:prob} if there does not exist $x \in D$ such that $f(x) \le
f(x^*)$ with $f_{j_0}(x) < f_{j_0}(x^*)$ for at least one index $j_0
\in \{1,\dots,m\}$.  Clearly, a Pareto optimum is a weak Pareto
optimal solution.

From now on, we assume the existence of a Slater point, i.e., we 
suppose that 
\begin{equation}
  \label{eq:relint}
  D^o := \{x \in \R^n : g(x) < 0\} \neq \emptyset.
\end{equation} 
We intend to find Pareto or weak Pareto solutions of problem
\eqref{eq:prob} by means of a barrier-type method for multiobjective
optimization, which produces its iterates in~$D^o$ and such that,
under certain mild conditions, will approximate optimal points laying
outside~$D^o$.

In our setting, we define an internal penalty function
for the feasible set $D$ as follows.

\begin{definition} 
  A \emph{multiobjective barrier function} for the set $D$ is a
  continuous function $B \colon D^o \to \R^m_+$ such
  that for any sequence $\{z^k\} \subset D^o$ with $g_i(z^k) \to 0$
  for an index $i \in \{1,...,p\}$, then $\lim_{k \to \infty} B_{j_i}
  (z^k) = +\infty$ for at least one index $j_i \in \{1,\dots,m\}$.
\end{definition}

\noindent Note that, when $m = 1$, we retrieve the classical notion of
barrier function used in scalar-valued optimization.  See, for
instance \cite[Section~5.1]{Ber16} or \cite[Section~9.4]{BSS06}.


\subsection{Examples of multiobjective barrier functions}
\label{ex:barrier} 

We now give some examples of multiobjective barrier functions for 
feasible sets~$D$ given by finitely many scalar inequalities, which
are extensions of  the inverse and the logarithmic barriers for
real-valued optimization. 

\begin{example}
  \label{ex:Bar1}
  If $p \leq m$, for $x \in D^o$, we define 
  \begin{displaymath}  
    B(x) := \Big(\frac{1}{-g_1(x)},\dots,\frac{1}{-g_p(x)},0,\dots,0 \Big)^\T. 
  \end{displaymath} 
  If $p > m$, for $x \in D^o$, we can take
  \begin{displaymath}  
    B(x) := \Big[\sum_{i=1}^p \frac{1}{-g_i(x)}\Big] e,
  \end{displaymath} 
  recalling that $e = (1,\dots,1)^\T \in \Re^m$. In this case, we can
  also take, for example,
  \begin{displaymath}
    B(x) := \Big(\frac{1}{-g_1(x)},\dots,\frac{1}{-g_{m-1}(x)},
    \sum_{i=m}^p \frac{1}{-g_i(x)} \Big)^\T.
  \end{displaymath}
\end{example} 
 
\begin{example}
  \label{ex:Bar2}
  Let us consider the case in which $-\log(-g_i(x))$ replaces  $\frac{1}{-
    g_i(x)}$ in the above example. Of course, functions as 
  \begin{equation}
    \label{eq:nonneg}
    B(x) := \Big[-\sum_{i=1}^p \log(-g_i(x))\Big] e,
  \end{equation} 
  may fail to be nonnegative. For our purposes, under certain
  circumstances, the use of such $B$ is equivalent to the use of a
  positive function of the form $B - \rho e$, where $\rho \in \R$ is
  constant. Indeed, as we will see in Section~\ref{sec:met}, the
  strategy we are about to propose, which is an extension of the
  classical barrier method, basically consists of somehow minimizing
  the objective function $f$ with the addition of a positive multiple
  of the barrier $B$. Therefore, the presence of the constant~$\rho$
  does not have any kind of influence when a minimizer of this
  perturbed objective function is computed.

  For instance, if $D$ is bounded, then $B$ given in~\eqref{eq:nonneg}
  is bounded from below, say $B(x) > \rho e$ for all $x \in D^o$, for
  some constant $\rho$. So, $\bar B(x) := B(x) - \rho e > 0$ for all
  $x \in D^o$ and, for $\tau > 0$, the minimizers of $f(x) + \tau \bar
  B(x) = f(x) + \tau (B(x) - \rho e)$ are the same as those of $f(x) +
  \tau B(x)$. (To be more precises, the minimization of the penalized
  objective will be carried on in $\R$, since, as we will see in
  Subsection~\ref{sec:alg}, this perturbation of $f$ will be composed
  with a real-valued monotonic function.)
\end{example}


\section{A barrier-type method for multiobjective optimization}
\label{sec:met}

In this section, we propose a barrier-type method for the constrained
multiobjective problem~\eqref{eq:prob}. We begin by presenting some
functions that will be used in the method.


\subsection{Auxiliary functions} 
\label{sec:prelim}

Here we introduce the auxiliary functions necessary for the
multiobjective-type barrier method. Essentially, they are continuous
and have some kind of monotonic behavior. First, let us recall a
couple of well-known notions (see~\cite{Jah03}).

For a closed set $C \subset \R^m$, a function $\Phi \colon C \to \R$
is called \emph{strictly increasing} or \emph{weakly-increasing}
(\emph{$w$-increasing}) in $C$ if, for all $u, v \in C$,
\begin{displaymath}
  u < v \quad \Rightarrow \quad \Phi(u) < \Phi(v),
\end{displaymath}
and $\Phi$ is called \emph{strongly increasing}
(\emph{$s$-increasing}) in $C$ if for all $u, v \in C$,
\begin{displaymath}
  u \le v \mbox{ and } u_{j_0} < v_{j_0} \, \mbox{ for at least one } 
  j_0 \quad \Rightarrow \quad \Phi(u) < \Phi(v).
\end{displaymath}
 
 
\noindent Clearly, an $s$-increasing function is also
$w$-increasing. Note that, for $n=m=1$, $s$ and $w$-increasing
functions are both increasing.

A simple example of a $w$-increasing function, which is not
$s$-increasing, in $C = \R^m$ is $\Phi \colon \R^m \to \R$, given by
$\Phi(u) : = \max_{i=1,\dots,m} \{u_i\}$.  A general example of
$w$-increasing function which is not $s$-increasing is the following:
$\Phi(u): = \sum_{i=1}^m \phi_i(u_i)$, where $\phi_i \colon \R \to \R$
is nondecreasing for $i=1,\dots,m$ and there exists at least one index
$i_0$ such that $\phi_{i_0}$ is increasing and not all $\phi_i$'s are
increasing functions.

An example of an $s$-increasing bounded function is $\Phi \colon \R^m
\to \R$, given by $\Phi(u): = \sum_{i=1}^m \arctan(u_i)$. More
generally, $\Phi(u): = \sum_{i=1}^m \phi_i(u_i)$, where $\phi_i \colon
\R \to \R$ is increasing for $i=1,\dots,m$, is an $s$-increasing
function (not necessarily bounded). For other examples of such
auxiliary functions, we refer to~\cite[Section~3]{FGDR16}.

It is also easy to see that if the function $\Phi$ is continuous and
$w$-increasing in $C$, then we have
\begin{equation}
\label{w-ineq}
  u \le v \quad \Rightarrow \quad \Phi(u) \le \Phi(v) 
  \quad \mbox{for any } u, v \in C.
\end{equation}


\noindent For the sake of completeness, we state a simple result for
these type of monotonic functions which relates optimality for the
scalar-valued problem $\min_{x \in D} \Phi(f(x))$ to weak Pareto and
Pareto optimality for the vector-valued problem~\eqref{eq:prob}. Note
that the continuity (of $f$ or of $\Phi$) is not needed.

\begin{lemma} 
  \label{L1} 
  \begin{itemize}
  \item[(a)] \label{L1(i)} If $\Phi$ is a $w$-increasing function in $f(D)$
    and $x^* \in \argmin_{x \in D} \Phi \big(f(x)\big)$, then $x^*$ is
    a weak-Pareto optimal solution for problem~\eqref{eq:prob}.
  \item[(b)] \label{L1(ii)} If $\Phi$ is an $s$-increasing function in
    $f(D)$ and $x^* \in \argmin_{x \in D} \Phi \big(f(x)\big)$, then
    $x^*$ is a Pareto optimal solution for problem~\eqref{eq:prob}.
  \end{itemize}
\end{lemma}

\begin{proof}
  The results follow from~\cite[Lemmas~5.14 and~5.24]{Jah03}.
\end{proof}




\subsection{The algorithm}  
\label{sec:alg}
 
Now, we define the \emph{multicriteria barrier method} for solving
problem~\eqref{eq:prob}, an extension of the classical procedure for
scalar-valued constrained optimization.

\begin{algorithm} The multicriteria barrier method (MBM):\\
  Take $B \colon D^o \to \R^m_+$ a multiobjective barrier for $D
  \subseteq \R^n$, $\Phi \colon \R^m \to \R$ a $w$-increasing
  continuous auxiliary function, and $\{\tau_k\} \subset \R_{++} =
  (0,+\infty)$ a sequence such that $\tau_{k+1} < \tau_k$ for all $k$
  and $\tau_k \to 0$. \\
  The method is iterative and generates a sequence $\{x^k\} \subset
  \R^n$ by solving the following scalar-valued optimization problems:
  \begin{equation} 
    \label{sup_prob}
    x^k \in \argmin_{x \in D^o} \Phi \Big(f(x) + \tau_k B(x)\Big),
    \quad k = 1,2,\dots
  \end{equation}
\end{algorithm}

\noindent The strong version of the method is formally identical to
the weak one, but with a continuous $s$-increasing auxiliary function
$\Phi \colon \R^m \to \R$.

From now on, for problem \eqref{eq:prob}, understood in both the weak
Pareto or Pareto senses, we assume that
\begin{equation} 
  \label{Hypot1}
  \inf_{x \in D^o} \Phi(f(x)) = \inf_{x \in D} \Phi(f(x)) \in \R,
\end{equation}
where, according to the case we seek Pareto or weak Pareto optima,
$\Phi$ is, respectively, a continuous strongly or weakly increasing
function. Observe that, in any case, by virtue of the monotonic
behavior of $\Phi$, this condition is an extension of the classical
hypotheses needed in the scalar case (i.e., $m$ = 1), which is given
by $\inf_{x \in D^o} f(x) = \inf_{x \in D} f(x)) > -\infty$.

Let us now list some comments and observations concerning both
versions of multicriteria barrier method.

\begin{enumerate}[leftmargin=*]
  
\item In order to guarantee the existence of $x^k$ for all $k$, we
  need some additional assumptions. One possibility is to apply the
  method for any functions $f$, $\Phi$ and $B$ such that $x \mapsto
  \Phi\big(f(x) + \tau_k B(x)\big)$ is ``coercive'' in $D^o$ for
  all~$k$, so it diverges to $+\infty$ whenever the argument
  approaches more and more to the boundary of~$D$. In such case,
  Weierstrass theorem shows that $\argmin_{x \in D^o} \Phi\big(f(x) +
  \tau_k B(x)\big) \neq \emptyset$ for all $k$.

\item Observe that whenever the data $f$ and $g$ are smooth, it is
  convenient to choose $B$ and $\Phi$ as smooth functions too, so
  that the subproblems~\eqref{sup_prob} can be solved by means of
  first-order scalar methods.

\item When $m=1$, any auxiliary function $\Phi \colon \R \to \R$ (weak
  or strong) is increasing, so, in this case, the
  update~\eqref{sup_prob} generates the same sequence as the classical
  scalar-valued barrier method.

\item In the scalar-valued setting, the behavior of the barrier~$B$
  close to the border of the feasible region seems a natural
  assumption for the welldefinedness of the method. Moreover, it is
  natural to expect that the presence of the barrier forces any
  unconstrained procedure used to compute the iterates to not
  choose~$x^k$ too close to the boundary of~$D$. In the multiobjective
  case, due to the presence of $\Phi$, which may be bounded, this is
  not so obvious; however, taking an auxiliary function~$\Phi$, which
  satisfies the condition $u_i \leq \Phi(u)$ for all $u \in \R^m$ and
  all $i=1,2,\dots,m$ (see examples of~$w$ or~$s$-type functions
  in~\cite{FGDR16}), it becomes clear that one can also expect this
  kind of behavior for any unconstrained procedure used to
  compute~$x^k$.  Anyway, for any~$m$, these subroutines must not be
  completely unconstrained, because they have to guarantee the
  feasibility of all $x^k$.

\item Let $\partial D$ be the boundary of set~$D$. Since, in general,
  $\argmin_{x \in D} \Phi(f(x)) \subset \partial D$, at least in the
  scalar-valued case, it is natural to ask the parameter sequence
  $\tau_k \to 0$, in order to compensate the fact that $B(z^k) \to
  +\infty$ for sequences $\{z^k\}$ which approach $\partial D$. In the
  multiobjective case, since $\Phi$ may be bounded, this condition
  does not seem so natural; nevertheless, if, again, one chooses an
  auxiliary function $\Phi$ such that $u_i \leq \Phi(u)$ for all $u
  \in \R^m$ and all $i=1,2,\dots,m$, we see that this prescription is
  also very reasonable.

  
\item This method inherits some features of its real-valued
  counterpart.  Firstly, we mention a drawback of MBM, namely that it
  does not have any kind of ``memory'', i.e., the former iterate needs
  not to be used to compute the current one. However, it seems natural
  to obtain~$x^k$ by initializing the subroutine used to solve
  subproblem~\eqref{sup_prob} with the previous iterate $x^{k-1}$.
  Secondly, a benefit of applying MBM is to replace a constrained
  vector-valued problem with a sequence of unconstrained scalar-valued
  ones.

\item Note that when $\Phi(u) = \max\{u_i\}$ and $B = (\hat
  B,\dots,\hat B)^\T$, where $\hat B$ is a real-valued barrier
  function for $D$, the use of MBM to \eqref{eq:prob} reduces to the
  application of the classical scalar barrier method to problem
  $\min_{x \in D} \max_{1 \leq i \leq m} \big\{f_i(x)\big\}$. One may
  then ask why we should use MBM. An answer to this question is that,
  MBM has more degrees of freedom: we do not always need to choose
  that specific $\Phi$ and neither a barrier of the type $B = (\hat
  B,\dots,\hat B)^\T$. Moreover, the choice of other barriers may be
  very useful whenever some of the given data components are in quite
  different scales. Indeed, in order to compensate this drawback, MBM
  can be implemented with appropriate barrier components $B_i$. Also,
  as we have already mentioned in item~2, when~\eqref{eq:prob} is
  smooth, a max type auxiliary function is definitely not the best
  choice for $\Phi$, since it will produce nonsmooth subproblems.

\item As in other extensions of classical scalar methods to the
  vectorial setting, under certain regularity conditions, all iterates
  of MBM are implicitly obtained by the application of the
  corresponding real-valued algorithm to a certain weighted
  scalarization. Let us see this assertion.  Assume that $f$ and $B=
  (\hat B,\dots,\hat B)^\T$ are differentiable $\R^m_+$-convex
  functions (i.e., $f_j$ and $B_j = \hat B$ are convex for all $j$),
  with $\hat B$ a scalar-valued penalty for $D$. Let $\Phi$ be defined
  by $\Phi(u) = \max_{i=1,\dots,m} \{ u_i \}$ for all $u \in \R^m$. It
  is a well-known fact (see, for example, \cite{HL01}) that $x^k$ is
  an unconstrained minimizer of $\max_{i=1,\dots,m} \{f_i(x) + \tau_k
  B_i(x)\}$ if and only if there exist positive scalars $\alpha_j =
  \alpha_j^{(k)}$, $j \in I(x^k)$, where
  \begin{displaymath}
    I(x^k) = \Big\{j
    \colon f_j(x^k) + \tau_k B_j(x^k) = \max_{i=1,\dots,m} \{f_i(x^k) +
    \tau_k B_i(x^k)\}\Big\},
  \end{displaymath}
  such that
  \begin{displaymath}
    \sum_{j \in I(x^k)} \alpha_j = 1 \,\,\, \mbox{ and }
    \sum_{i\in I(x^k)} \alpha_i \Big(\nabla f_i(x^k) + \tau_k \nabla
    B_i(x^k)\Big) = 0.
  \end{displaymath}
  Taking $\alpha_j = 0$ for all $j \in \{1,\dots,m\} \setminus
  I(x^k)$, we have
  \begin{equation}
    \label{Eq:5}
    \sum_{i=1}^m \alpha_i \Big(\nabla f_i(x^k) + \tau_k \nabla
    B_i(x^k)\Big) = 0. 
  \end{equation} 
  But, under our assumptions, this equality is also a necessary and
  sufficient optimality condition for the penalized scalar-valued
  convex function
  \begin{displaymath}
    x \mapsto \sum_{i=1}^m \alpha_i \Big(f_i(x) + \tau_k B_i(x)\Big) =
    \langle \alpha, f(x) \rangle  + \tau_k \hat B(x),
  \end{displaymath}
  where we used that $ B_i = \hat B$ for $i=1,\dots,m$, $\alpha
  :=(\alpha_1,\dots,\alpha_m)^\T$ and $\sum_{i=1}^m \alpha_i = 1$.
  That is to say,
  \begin{displaymath}
    x^k \in \argmin_{x \in \R^n} \langle \alpha, f(x) \rangle + \tau_k
    \hat B(x),
  \end{displaymath}
  which in turn means that $\{x^k\}$ can be obtained via the
  application of the classical (scalar) barrier method to the
  real-valued function $x \mapsto \langle \alpha, f(x) \rangle $, a
  weighted scalarization of the vector-valued objective $f$, with
  weighting vector given by $\alpha = \alpha^k \in \R^m_+$, using the
  scalar-valued barrier $\hat B(x)$ for $D$ and $\{\tau_k\}$ as the
  parameter sequence.  Of course, a priori, we do not know the
  nonnegative weights $\alpha_1,\dots,\alpha_m$ which add up one and
  satisfy the Lagrangian equation \eqref{Eq:5}.

\end{enumerate}


\section{Convergence analysis of MBM}
\label{sec:conv}

We begin this section by showing some of elementary properties of
sequences produced by both versions of MBM which will be needed in
the sequel.


\begin{lemma} 
  \label{L2}
  Let $\{x^k\} \subset \R^n$ be a sequence generated by the weak or
  strong version of MBM, implemented with a barrier function $B \colon
  D^o \to \R^m_+$, a parameters sequence $\{\tau_k\} \subset \R_{++}$
  and an auxiliary function $\Phi \colon \R^m \to \R$. Define $\Phi^*
  : = \inf _{x \in D} \Phi(f(x))$ and $\Phi_k: = \Phi(f(x^k) +
  \tau_kB(x^k))$ for all~$k$. Then, the following statements hold:
  \begin{itemize}
  \item[(a)] \label{L2(i)} For all $k = 1,2,\dots$, we have
   \begin{displaymath}
    \Phi^*  \le \Phi_{k+1} \le  \Phi_k.
   \end{displaymath}
 \item[(b)] \label{L2(ii)} There exists $\eta \in \R$ such that
   \begin{displaymath}
     \lim_{k \to \infty} \Phi_k = \eta.
   \end{displaymath}
   In particular, $\Phi^* \leq \eta$.      
 \end{itemize} 
\end{lemma}

\begin{proof} 
  \begin{itemize} 
  \item[(a)] Using the definition of $\Phi_k$ and the fact that $x^k \in
    D^o$ for all $k$, we obtain
    \begin{align*}
      \Phi^*   
      & =  \inf_{x \in D} \Phi(f(x)) \\
      & =  \inf_{x \in D^o} \Phi(f(x)) \\
      & \le \Phi(f(x^{k+1})) \\
      & \le  \Phi(f(x^{k+1}) + \tau_{k+1}B(x^{k+1})) \\
      & = \Phi_{k+1} \\
      & = \min_{x \in D^o} \Phi\big(f(x) + \tau_{k+1} B(x)\big) \\
      & \le \Phi\big(f(x^k) + \tau_{k+1} B(x^k)\big) \\
      & \le \Phi\big(f(x^k) + \tau_k B(x^k)\big) \\ 
      & = \Phi_k,
    \end{align*} 
    where the second equality follows from~\eqref{Hypot1}, the second
    and the last inequalities follow from the facts that $\Phi$ is
    w-increasing, $0 \le B(x)$ for any $x$, $0 < \tau_{k+1} < \tau_k$
    for all $k$ and~\eqref{w-ineq}. Note that this proof holds for
    $\Phi$ weakly or strongly increasing.
  \item[(b)] By item~(a) and~\eqref{Hypot1}, we have that $\big
    \{\Phi_k\big \} \subset \R$ is a nonincreasing bounded from below
    sequence, so, as $k \to \infty$, it converges to some $\eta \in
    \R$.
  \end{itemize} 
\end{proof}

Now, we state and prove an extension of a classical convergence result
for real-valued optimization, which establishes that accumulation
points, if any, of a sequence generated by the barrier method are
optima for the original constrained minimization problem.
 
\begin{theorem}    
  \label{T1} 
  Let $\{x^k\} \subset \R^n$ be a sequence generated by the weak or
  strong version of~MBM, implemented with an
  auxiliary function $\Phi \colon \R^m \to \R$, a multiobjective
  barrier function $B \colon D^o \to \R^m_+$ and a
  sequence of parameters $\{\tau_k\} \subset \R_{++}$. If $\bar x \in
  D$ is an accumulation point of $\{x^k\}$, then $\bar x \in
  \argmin_{x \in D} \Phi(f(x))$. Moreover, if $\Phi$ is
  $w$-increasing, then~$\bar x$ is a weak Pareto optimum for problem
  \eqref{eq:prob}; if~$\Phi$ is $s$-increasing, then~$\bar x$ is a
  Pareto optimal solution for~\eqref{eq:prob}.
\end{theorem}  

\begin{proof} Let us first consider the case in which MBM is
  implemented with a weakly-increasing auxiliary function $\Phi$.
  From Lemma~\ref{L2}, there exists a real number $\eta \in \R$ such
  that $\Phi_k \to \eta$, where $\Phi_k = \Phi(f(x^k) + \tau_k
  B(x^k))$. As in the same lemma, call $\Phi^* = \inf _{x \in D}
  \Phi(f(x))$. We claim that $\Phi^* = \eta$.  If not, then, again by
  Lemma~\ref{L2}, we have
  \begin{equation}
    \label{ineq} 
    \Phi^* < \eta.
  \end{equation}
  Since, by~\eqref{Hypot1}, $\Phi^* = \inf_{x \in D}\Phi(f(x)) =
  \inf_{x \in D^o} \Phi(f(x))$ and $\Phi^* < \eta$, there exists
  $\tilde x \in D^o$ such that $\Phi^* < \Phi(f(\tilde x)) <
  \eta$. Hence, from the continuity of $\Phi$ and the fact that
  $\tau_k \to 0$, for $k$ large enough we have that
  \[
  \Phi(f(\tilde x) + \tau_k B(\tilde x)) < \eta.
  \]
  From the fact that $\tilde x \in D^o$, it follows that $\Phi_k =
  \min_{x \in D^o} \Phi (f(x) + \tau_kB(x)) \leq \Phi(f(\tilde x) +
  \tau_k B(\tilde x))$, which combined with the above inequality leads
  to
  \[
  \Phi_k < \eta
  \]
  for $k$ large enough, in contradiction with the results of
  Lemma~\ref{L2}.

  Whence \eqref{ineq} does not hold and so
  \begin{equation}
    \label{eq1}
    \Phi^* = \eta. 
  \end{equation}
  Now let $\{x^{k_j}\}$ be a subsequence of $\{x^k\}$ such that
  $\lim_{j \to \infty} x^{k_j} = \bar x.$ Since $x^{k_j} \in D^o$ for
  all $j=1,2,\dots$, we have $g(x^{k_j}) < 0$ for all $j$ and, by the
  continuity of $g$, $g(\bar x) \le 0$, i.e., $\bar x \in D$. Assume
  that $\bar x \notin \argmin_{x \in D} \Phi(f(x))$; from the
  feasibility of $\bar x$ and the definition of $\Phi^*$, this means
  that
  \begin{equation}
    \label{ineq2}
    \Phi(f(\bar x)) > \Phi^*. 
  \end{equation}
  From \eqref{eq1} we get that
  \begin{equation}
    \label{lim1}
    \lim_{j\to \infty} \Phi(f(x^{k_j}) + \tau_{k_j}B(x^{k_j})) -
    \Phi^* = \lim_{j \to \infty}  \Phi_{k_j} -
    \Phi^* = \eta - \Phi^* = 0.
  \end{equation}
  But, combining the fact that
  $\tau_{k_j}B(x^{k_j}) \ge 0$ for all $j=1,2,\dots$ with~\eqref{w-ineq}, and using the
  continuity of $\Phi(f(\cdot))$, the definition of $\Phi^*$,
  as well as~\eqref{ineq2}, we also get
  \begin{equation}
    \lim_{j\to \infty} \Phi(f(x^{k_j}) + \tau_{k_j}B(x^{k_j})) -
    \Phi^* \ge \lim_{j\to \infty} \Phi(f(x^{k_j})) -
    \Phi^* = \Phi(f(\bar x))  - \Phi^* > 0,
  \end{equation}
  which contradicts~\eqref{lim1}.
  
  Hence~\eqref{ineq2} is not true and so $\Phi(f(\bar x)) = \Phi^* =
  \inf_{x \in D} \Phi(f(x))$, that is to say $\bar x \in \argmin_{x\in
    D}\Phi(f(x))$ and the result follows from item~(a) of
  Lemma~\ref{L1}. The proof for the case in which MBM is implemented
  with an $s$-increasing $\Phi$ is almost identical, and it concludes
  using item~(b) of Lemma~\ref{L1}.
\end{proof} 

The following result establishes that both versions of MBM are
convergent whenever $\Phi \circ f := \Phi(f(\cdot))$ has a strict
minimizer in $D$, which happens, for instance, if this composition is
strictly convex and coercive.

\begin{corollary} 
  \label{cor:gen_mbm}
  Let $\{x^k\} \subset \R^n$ be generated by the weak or strong
  version of MBM, with auxiliary function $\Phi \colon \R^m \to \R$
  such that $x \mapsto \Phi\big(f(x)\big)$ has a strict
  minimizer~$\bar x$ in~$D$, i.e., $\argmin_{x \in D} \Phi(f(x)) =
  \{\bar x\}$. If we assume that $\{x^k\}$ has an accumulation point,
  then $x^k \to \bar x$. Furthermore, if $\Phi$ is $w$-increasing,
  then $\bar x$ is a weak Pareto optimal solution for~\eqref{eq:prob};
  if $\Phi$ is $s$-increasing, then $\bar x$ is a Pareto optimum
  for~\eqref{eq:prob}.
\end{corollary}

\begin{proof}
  Let us just prove the case in which $\{x^k\}$ is generated by the
  weak version of MBM. From Theorem~\ref{T1} and the strict
  optimality of $\bar x \in D$, any subsequential limit of $\{x^k\}$
  is equal to $\bar x$ and so $x^k \to \bar x$, a weak Pareto optimal
  solution for~\eqref{eq:prob}.  The convergence result for the strong
  version of MBM also follows from Theorem~\ref{T1}.
\end{proof}


\subsection{Local implementation of MBM and its convergence}
\label{sec:loc}

We now sketch a practical implementation of both, strong and weak,
versions of MBM. Take $B \colon \R^n \to \R^m_+$, a barrier for $D$, a
$w$-increasing continuous auxiliary function $\Phi \colon \R^m \to
\R$, a decreasingly convergent to zero sequence of positive penalty
parameters $\{\tau_k\}$ and $V \subset \R^n$, a compact set with
nonempty interior (e.g., a closed ball with positive radius). Let
\begin{equation}
  \label{loc_sub_prob}
  x^k \in \argmin_{x \in D^o \cap {\rm int}(V)} \Phi\big(f(x) + \tau_k B(x)\big),
  \quad k = 1,2,\dots
\end{equation}
where int($V$) stands for the  topological interior of $V$. 

\noindent The strong version of this local implementation of the
method is formally identical to the weak one, but with a
$s$-increasing continuous auxiliary function $\Phi$ instead of a
$w$-increasing one.

From now on, we assume that
\begin{equation}
  \label{Hypot1Local}
  \inf_{x \in D^o \cap \, {\rm int} (V)} \Phi(f(x)) 
  = \inf_{x \in D \cap V} \Phi(f(x)) \in \R,
\end{equation}
where $\Phi$ is a strongly or weakly-increasing continuous function if
we are, respectively, seeking Pareto or weak Pareto optima.

Assuming the existence of a strict local minimizer of $\Phi(f(\cdot))$
within $D$ (i.e., a feasible $\bar x$ such that $\Phi(f(\bar x)) <
\Phi(f(x))$ for any feasible $x$ in a vicinity of $\bar x$), we will
prove that both the weak and the strong versions of the local MBM
are fully convergent to a weak Pareto and a Pareto optimal solutions,
respectively.

\begin{theorem} 
  \label{T2}
  Suppose that $\Phi \colon \R^m \to \R$ is a weakly or strongly
  increasing continuous function and that $\bar x \in D$ is a strict
  local minimizer of $\Phi(f(\cdot))$ in $D$, say $\bar x = \argmin_{x
    \in D \cap U} \Phi\big(f(x)\big)$, for some vicinity $U \subset
  \R^n$ of $\bar x$. Let $\{x^k\} \subset \R^n$ be a sequence
  generated by the local MBM version, implemented with a barrier
  function $B \colon \R^n \to \R^m_+$, a parameters sequence
  $\{\tau_k\} \subset \R_{++}$, the auxiliary function $\Phi$ and $V$,
  a compact subset of $U$ with nonempty interior. Then $x^k \to \bar
  x$ and, moreover, $\bar x$ is a weak Pareto optimal solution for
  problem~\eqref{eq:prob}.  If $\Phi$ is a $s$-increasing auxiliary
  function, then $x^k \to \bar x$ and $\bar x$ is a Pareto optimum for
  problem \eqref{eq:prob}.
\end{theorem}

\begin{proof}
  Since $V$ is compact, the sequence $\{x^k\} \subset V$ has an
  accumulation point $\tilde x \in V$. Note that
  condition~\eqref{Hypot1Local} allows us to prove a local version of
  Lemma~\ref{L2}, so as in the proof of Theorem~\ref{T1}, we see that
  $\tilde x \in \argmin_{x \in D \cap V} \Phi(f(x))$. Since $\bar x$
  is the unique minimizer of $\Phi(f(\cdot))$ within $D \cap U$ and $V
  \subset U$, we conclude that $\tilde x = \bar x$; therefore, the
  unique accumulation point of $\{x^k\}$ is $\bar x$ and the proof is
  complete. Now, from Lemma~\ref{L1}, $\bar x$ is a weak Pareto or a Pareto
  optimum, whether $\Phi$ is, respectively, a $w$-increasing or a
  $s$-increasing auxiliary function.
\end{proof}

Local convergence may also be true under a weaker condition than the
existence of a strict local minimizer, namely whenever $S_U : =
\argmin_{x \in D \cap U} \Phi(f(x))$ is an isolated set of $S : =
\bigcup_{W \subset \R^n} S_W$, where $ S_W : = \{ \bar x \in \R^n \colon
\Phi(f(\bar x)) = \inf_{x \in D \cap W} \Phi(f(x))\}$, with $W$
compact subset of $\R^n$ such that int$(W) \neq \emptyset$, which
means that there exists a closed set $G \subset \R^n$, such that
$\emptyset \neq S_U \subset$ int$(G)$ and $G \setminus S_U \subset
\R^n \setminus S$.


\section{MBM applications}
\label{sec:app}

Let us now show a very simple \emph{ad hoc} instance of
problem~\eqref{eq:prob} where the weighting method, a quite popular
strategy for solving multicriteria problems, fails to furnish optima
in an arbitrary large set of weights, while MBM provides an optimal
solution.

Recall that the weighting method consists of minimizing a convex
combination of the objective components on the feasible set. So for $m
= 2$, it just requires a single weight $\alpha \in [0,1]$ (the other
one is $1 - \alpha$). For a given tolerance $\varepsilon \in (0,1)$,
we will show that, when applied to~\eqref{eq:prob} for a certain $f =
(f_1,f_2)^\T$ and a particular constraint set~$D$, this method fails to
have optimal solutions in a set of parameters whose length is greater
than $1 - \varepsilon$.

\begin{example}
  Let $\varepsilon \in (0,1)$. Consider $f \colon \Re \to \Re^2$ be
  given by $f(t) := (t,-a t)^\T$, where $a >
  \frac{1-\varepsilon}{\varepsilon}$, with $D := \{t \in \R \colon t
  \geq 0\}$.  So, the multiobjective problem can be written as
  \begin{equation}
    \label{ex_prob} 
    \min_{-t \leq 0} \, f(t). 
  \end{equation}
  The set of Pareto (weak or not) optima is $S = D = [0,+\infty)$.
  The weighting method prescribes to solve the following scalar
  optimization problem:
  \[
  \min_{t \geq 0} \, \big\langle (\alpha,1-\alpha), (t,-a t) \big\rangle
  = \min_{t \geq 0} \, [\alpha (1 + a) - a]t.
  \]
  Clearly, the above problem has an optimal solution if, and only if, $
  \alpha (1 + a) - a \geq 0$, which, in turn, is equivalent to $\alpha
  \geq \frac{a}{1+a}$. This means, that the weighting method works
  properly only in an interval with length $1 - \frac{a}{1+a} =
  \frac{1}{1 + a}$.

  Therefore, since $a > \frac{1-\varepsilon}{\varepsilon}$, the
  weighting method gives us an optimum in an interval of weights of
  length smaller than $\varepsilon > 0$ or, in other words, it fails
  taking the parameter~$\alpha$ in a subinterval of $[0,1]$ with
  length greater than $1 - \varepsilon$.

  On the other hand, if we implement MBM with $B \colon D^o \to
  \R^2_+$, given by $B(t) : = (\frac{1}{t},\frac{1}{t})^\T$, where
  $D^o = \{t \in \R \colon t > 0\}$, $\tau_k = \frac{1}{k}$ for all $k
  \in \mathbb{N}$ and $\Phi \colon \R^2 \to \R$, given by $\Phi(u) :=
  \max\{u_1,u_2\}$, it is easy to see that
  \[
  t_k : = k^{-1/2} \in \argmin_{t \in D^o} \, \Phi\big(f(t) + \tau_k
  B(t)\big) = \argmin_{t \in D^o} \max \left\{ t+ \frac{1}{kt}, -a t +
  \frac{1}{kt} \right\} =t + \frac{1}{kt} \,.
  \] 
  Note that $\Phi(f(t)) = \max \{t,-at\} = t$ for all $t \geq 0$, so
  we have that $\argmin_{t \in D} \, \Phi(f(t)) = \{0\}$. Since
  $-\infty < 0 = \inf_{t \in D} \, \Phi(f(t)) = \inf_{t \in D^o} \,
  \Phi(f(t))$ and also $t_k \to 0$, by Corollary~\ref{cor:gen_mbm}, we
  have that MBM furnishes $t^* = 0$, an optimal point for
  problem~\eqref{ex_prob}.
\end{example}

We now exhibit a very simple instance of problem~\eqref{eq:prob}, for
which there exists a family of auxiliary functions
$\{\Phi_\omega\}_{\omega \in \Omega}$, $\Omega \subset \R^m$, such
that, using local MBM implemented with these functions, any
multiobjective barrier function for the feasible set and any positive
and decreasing to zero parameters sequence, by varying $\omega \in
\Omega$, we can retrieve the whole optimal set.

First, observe that, for any $\omega \in \R^m$, the function
$\Phi_\omega \colon \R^m \to \R$, given by $\Phi_\omega (u) :=
\max_{i=1,\dots,m}\{u_i + \omega_i\}$ is $w$-increasing and
continuous, so it can be used as an auxiliary function.

\begin{example}
  Consider $n=1$, $m=2$, $D =[-2,+\infty)$ and $f \colon \R \to \R^2$
  defined by $f_1(t) := t^2 +1$, $f_2(t) := t^2-2t+1$. Clearly, in the
  interval $[0,1]$, whenever $f_2$ decreases, $f_1$ increases and, in
  $D$, this happens only in this interval; that is to say, $[0,1]$ is
  the optimal set for problem~\eqref{eq:prob} with these data.
  
  We will apply MBM with $\Phi_\omega(u) = \max_{i=1,2}\{u_i +
  \omega_i\}$, $\omega \in \R^2$, a continuous $w$-increasing
  function.  First of all, note that $f_1(t) \le f_2(t)$ if and only
  if $t \le 0$, so $t \mapsto \Phi_\omega(f(t))$, with $\omega =
  (0,0)^\T$, has a strict minimizer at $t = 0$. Now, let us investigate
  $\argmin_{t \in D} \Phi_\omega(f(t))$ not just for $\omega =
  (0,0)^\T$, but for $\omega_\alpha : = (\alpha,0)^\T$, with $\alpha \in
  [-2,0]$.

  It is easy to see that $x \mapsto \Phi_{\omega_\alpha}(f(x))$ has a
  unique minimizer in $[0,1]$ at the sole point $t_\alpha \in [0,1]$
  where $f_1(t_\alpha) +\alpha = f_2( t_\alpha)$, that is to say
  $t_\alpha = -\frac{\alpha}{2}$.  Therefore,
  \[
  \bigcup_{-2 \leq \alpha \leq 0} \argmin_{t \in D}
  \Phi_{\omega_\alpha} (f(t)) = [0,1].  
  \]
  Applying Theorem~\ref{T2}, with the auxiliary function
  $\Phi_{\omega_\alpha}$, where $\omega_\alpha \in \Omega : = [-2,0]
  \times \{0\}$, any multiobjective barrier function $B$ for $D$, as
  well as any $V_\alpha \subset \R$ compact vicinity of $t_\alpha$ and
  any parameter sequence $\{\tau_k\} \subset \R_{++}$ such that
  $\tau_{k+1} < \tau_k$ and $\tau_k \to 0$, the generated sequence
  $\{t^k\}$ converges to $t_\alpha$.  This means that, using MBM
  with all  auxiliary functions of the family
  $\{\Phi_{\omega_\alpha}\}_{\alpha \in [-2,0]}$, we retrieve the
  whole optimal set of the multiobjective problem $\min_{-t \leq 2}
  f(t) = (t^2+1, t^2-2t+1)^\T$.

  Of course, this is also an {\it ad hoc} example, but it may be
  useful in order to investigate when do we have auxiliary functions
  families such that by varying the parameters we can obtain the whole
  optimal set by means of the sequences produced by MBM.
\end{example}


\section{Final remarks}
\label{sec:final}

In this work we developed an extension for the vector-valued setting
of the classical internal penalty method for single optimization. As
expected, the convergence results are generalizations of those which
hold when the method is applied to scalar problems. For future
research, we leave the study of conditions which guarantee the
existence of auxiliary functions families such that by varying the
parameters, the whole Pareto (weak or not) frontier can be generated
applying MBM implemented with all those functions. Example 5.1
together with Theorem 4.4 may shed some light on that matter. Indeed,
the fact that MBM only converges to minimizers of the scalar
representations induced by the auxiliary functions suggest that this
apparent drawback of the method can be useful in order to study
necessary or sufficient conditions for the existence of such families.
Another interesting thing to analyze in the multicriteria setting is
the developing of mixed interior-exterior penalty methods. Finally,
the existence of a generalization of MBM to the vector optimization
case, or even to the variable order setting, may deserve some
attention.





\bibliographystyle{plain}
\bibliography{references}


\end{document}